\documentclass[10pt]{amsart}
\usepackage{enumerate,amsmath,amssymb,latexsym,
amsfonts, amsthm, amscd, MnSymbol}

\theoremstyle{plain}

\newtheorem{theorem}{Theorem}

\numberwithin{equation}{section}

\newcommand{\CC}{C_Q}

\addtocounter{section}{-1}

\begin{document}

\title {Q-tableaux for Implicational Propositional Calculus}

\date{}

\author[P.L. Robinson]{P.L. Robinson}

\address{Department of Mathematics \\ University of Florida \\ Gainesville FL 32611  USA }

\email[]{paulr@ufl.edu}

 \keywords{}

\begin{abstract}

We study $Q$-tableaux and axiom systems that they engender, producing a new proof that the Implicational Propositional Calculus is complete. 

\end{abstract}

\maketitle

\section{Introduction} 

In a recent paper [4] we showed how completeness of the Implicational Propositional Calculus (IPC) may be established by means of `dual $Q$-tableaux' and their associated axiom systems; here, we study `$Q$-tableaux' themselves and the axiom systems to which they give rise. There are interesting differences between the two types of axiom systems: on the one hand, those arising from dual $Q$-tableaux are based on disjunction, which may be defined within IPC; on the other hand, those arising from $Q$-tableaux are based on conjunction, which only appears in IPC by (partial or residual) proxy. Among other technical differences, our proof that theorems of the $Q$-tableau axiom systems are provable within IPC assigns a much more pervasive r\^ole to the Peirce axiom scheme as a weak law of the excluded middle. 

\section{Q-tableaux for IPC}

\medbreak 

The Implicational Propositional Calculus (IPC) has a single connective ($\supset$) and a single inference rule (modus ponens or MP) with three axiom schemes: \par 
(${\rm IPC}_1$) \; \; $X \supset (Y \supset X)$\par 
(${\rm IPC}_2$) \; \; $[X \supset (Y \supset Z)] \supset [(X \supset Y) \supset (X \supset Z)]$ \par 
(Peirce) \;  $[(X \supset Y) \supset X] \supset X$. \\
As usual, $\vdash$ will signify deducibility within IPC; further, $\mathbb{T}$ will denote the set of theorems of IPC. In particular, the statements $\vdash X$ and $X \in \mathbb{T}$ are effectively synonymous. We remark that the Deduction Theorem (DT) and Hypothetical Syllogism (HS) are valid in IPC as derived inference rules; they will be used (perhaps silently) throughout this paper. Exercises 6.3-6.5 of [2] provide a convenient do-it-yourself introduction to IPC. 

\medbreak 

Fix a (well-formed) formula $Q$ of IPC and when $Z$ is any IPC formula write
$$Q Z := Q(Z) := Z \supset Q$$
so that $QQ Z = (Z \supset Q) \supset Q$ and so forth. 

\medbreak 

\begin{theorem} \label{Robbin} 
Each of the following is an IPC theorem scheme: \par
{\rm (1)} $(X \supset Y) \supset [(Y \supset Z) \supset (X \supset Z)]$ \par 
{\rm (2)} $(X \supset Y) \supset (QY \supset QX)$ \par 
{\rm (3)} $X \supset QQX$ \par
{\rm (4)} $QQQX \supset QX$ \par 
{\rm (5)} $QQY \supset QQ(X \supset Y)$ \par 
{\rm (6)} $QQX \supset [QY \supset Q(X \supset Y)]$ \par 
{\rm (7)} $QX \supset QQ(X \supset Y)$ \par 
{\rm (8)} $(QX \supset Y) \supset [(QQX \supset Y) \supset QQY].$ 
\end{theorem} 

\begin{proof} 
This is Exercise 6.3 in [2] so the proof is DIY.  The only part requiring Peirce is (7) as noted in [2]; more than this, Peirce follows by MP from (7) with $Q = X$ and the fact that $X \supset X \in \mathbb{T}$. 
\end{proof} 

\medbreak 

Disjunction ($\vee$) may be defined within IPC as an abbreviation: thus, 
$$X \vee Y := (X \supset Y) \supset Y.$$
This has the expected properties. For instance, $X \vdash X \vee Y$ (by MP and DT: $X, X \supset Y \vdash Y$ so $X \vdash (X \supset Y) \supset Y$) and $Y \vdash X \vee Y$ (by MP and ${\rm IPC}_1$). Moreover, the Peirce axiom scheme guarantees the following complementary property. 

\medbreak 

\begin{theorem} \label{Elimination}
If $X \vdash Z$ and $Y \vdash Z$ then $X \vee Y \vdash Z.$ 
\end{theorem} 

\begin{proof} 
This is Theorem 3 in [3]. 
\end{proof} 

\medbreak 

As an immediate consequence, $\vee$ is `commutative' in the sense that $Y \vee X \vdash X \vee Y$. As a slightly less immediate consequence, $\vee$ is `associative' in the sense that $X \vee (Y \vee Z) \vdash (X \vee Y) \vee Z$ and vice versa. As another consequence, we may rewrite the Peirce axiom scheme equivalently as a weak `law of the excluded middle'; we state this fact as a theorem, primarily for ease of reference. 

\medbreak 

\begin{theorem} \label{middle}
If $Q$ and $Z$ are IPC formulas then $Q Z \vee Z$ is a theorem of IPC. 
\end{theorem} 

\begin{proof} 
Rewrite! Thus: 
$$Q Z \vee Z = (Q Z \supset Z) \supset Z = [(Z \supset Q) \supset Z] \supset Z.$$
\end{proof} 

\medbreak 

\noindent
{\bf Remark:} We may use this to infer from $Q Z \supset W \in \mathbb{T}$ and $Z \supset W \in \mathbb{T}$ that $W \in \mathbb{T}$. In fact, if $Q Z \supset W \in \mathbb{T}$ and $Z \supset W \in \mathbb{T}$ then $Q Z \vdash W$ and $Z \vdash W$ by MP so that $Q Z \vee Z \vdash W$ by Theorem \ref{Elimination} and $(Q Z \vee Z) \supset W \in \mathbb{T}$ by DT; now MP and Theorem \ref{middle} place $W$ in $\mathbb{T}$. 

\medbreak 

We shall have need of the following extension to Theorem \ref{Robbin}. 

\medbreak 

\begin{theorem} \label{extra}
Each of the following is an IPC theorem scheme: \par 
$({\rm A}_0)\; \; $ $Q (X \supset Y) \supset QQ X;$ \par 
$({\rm A}_1)\; \; $ $Q (X \supset Y) \supset Q Y;$ \par 
$({\rm B})$\; \;  $QQ (X \supset Y) \supset (QQ Y \vee Q X).$
\end{theorem} 

\begin{proof} 
$({\rm A}_0)$ From Theorem \ref{Robbin} part (7) we have $Q X \supset QQ (X \supset Y) \in \mathbb{T}$ whence by Theorem \ref{Robbin} part (2) and MP we have $QQQ (X \supset Y) \supset QQ X \in \mathbb{T}$. Theorem \ref{Robbin} part (3) gives us $Q (X \supset Y) \supset QQQ (X \supset Y) \in \mathbb{T}$ and an application of HS gives us $Q (X \supset Y) \supset QQ X \in \mathbb{T}$. \medbreak 
$({\rm A}_1)$ Axiom scheme $({\rm IPC}_1)$ gives us $Y \supset (X \supset Y) \in \mathbb{T}$ and Theorem \ref{Robbin} part (2) gives us $[Y \supset (X \supset Y)] \supset [Q (X \supset Y) \supset Q Y] \in \mathbb{T}$; by MP we deduce that $Q (X \supset Y) \supset Q Y \in \mathbb{T}$. \medbreak 
$({\rm B})$ Note that $QQ (X \supset Y) = [(X \supset Y) \supset Q] \supset Q = (X \supset Y) \vee Q$. We shall prove separately that $X \supset Y \vdash QQ Y \vee Q X$ and that $Q \vdash QQ Y \vee Q X$; an application of Theorem \ref{Elimination} will then conclude the argument. {\it Proof of} $X \supset Y \vdash QQ Y \vee Q X$: Assume $X \supset Y$, $QQ Y \supset Q X$, $X$. Successive applications of MP yield: $Y$ (from $X$ and $X \supset Y$); $QQ Y$ (from $Y$ and $Y \supset QQ Y$ in Theorem \ref{Robbin} part (3)); $Q X$ (from $QQ Y$ and $QQ Y \supset Q X$); $Q$ (from $X$ and $X \supset Q = Q X$). This proves that  $X \supset Y, (QQ Y \supset Q X), X \vdash Q$ and two applications of DT yield $X \supset Y \vdash (QQ Y \supset Q X) \supset QX = QQ Y \vee Q X$. {\it Proof of} $Q \vdash QQ Y \vee Q X$: This is easy: axiom scheme (${\rm IPC}_1$) gives $Q \vdash X \supset Q = Q X$; now $Q, (QQY \supset Q X) \vdash Q X$ so that $Q \vdash (QQ Y \supset  Q X) \supset Q X = QQ Y \vee Q X$ by DT. 
\end{proof} 

\medbreak 

In fact, $(QQ Y \vee Q X) \supset QQ (X \supset Y) \in \mathbb{T}$ too: indeed, Theorem \ref{Robbin} part (5) tells us that $QQ Y \vdash QQ (X \supset Y)$ while Theorem \ref{Robbin} part (7) tells us that $Q X \vdash QQ (X \supset Y)$; all that remains is to invoke Theorem \ref{Elimination} again. 

\medbreak 

Conjunction ($\wedge$) itself may not be definable within IPC, but a shadow of conjunction does exist and this shadow serves our purposes. Our purposes require that within IPC there be available formulas that serve as proxies for expressions of the classical form $\sim (Z_N \wedge \cdots \wedge Z_0)$ where $\sim$ signifies negation. It is abundantly clear from the foregoing development (in particular, Theorem \ref{Robbin} and Theorem \ref{middle}) that the formula $Q Z = Z \supset Q$ has properties akin to those of the negation $\sim Z$; indeed, the framework for classical Propositional Calculus presented in [1] includes a propositional constant $\mathfrak{f}$ (for falsity) and defines $\sim Z$ to be $Z \supset \mathfrak{f}$. Taking into account this function of $Q$ in manufacturing a partial substitute for negation, along with the classical exportation and importation rules, we accordingly make the following definition.

\medbreak 

When $\theta = (Z_N, \dots , Z_0)$ is a sequence of IPC formulas, we define 
$$\CC (\theta) := Z_N \supset ( \cdots ( Z_0 \supset Q) \cdots )$$
where $\CC$ suggests negated conjunction. For convenience, we may omit brackets and write simply 
$$\CC(\theta) = Z_N \supset \cdots \supset Z_0 \supset Q$$
with the understanding that brackets are as displayed above. Observe at once that if $W$ is also an IPC formula and $W, \theta$ stands for the sequence $(W, Z_N, \dots , Z_0)$ then 
$$\CC(W, \theta) = W \supset \CC (\theta)$$
which observation is of course the essence of a formal inductive definition of $\CC$ starting from $\CC (Z_0) = Z_0 \supset Q.$

\medbreak 

The following property of this construction will be needed later. 

\begin{theorem} \label{vdash}
If $0 \leqslant n \leqslant N \in \mathbb{N}$ then $Q Z_n \vdash \CC (Z_N, \dots , Z_0)$. 
\end{theorem} 

\begin{proof} 
We write $C_N = \CC (Z_N, \dots , Z_0)$ for convenience and break the proof into stages.\\
(1) If $0 \leqslant n \leqslant N \in \mathbb{N}$ then $C_n \vdash C_N$. [$C_n \vdash Z_{n + 1} \supset C_n = C_{n + 1}$ is an instance of ${\rm IPC}_1$.] \\
(2) If $N \in \mathbb{N}$ then $Q \vdash C_N$. [$Q \vdash Z_0 \supset Q = C_0$ is an instance of ${\rm IPC}_1$; now invoke (1).]\\
(3) If $0 \leqslant n \in \mathbb{N}$ then $Q Z_n \vdash C_n$. [The base case $n = 0$ is plain: $Q Z_0 = Z_0 \supset Q = C_0$. For the inductive step, hypothesize $Q Z_n \vdash C_n$. From $Q Z_{n + 1} = Z_{n + 1} \supset Q$ and $Z_{n + 1}$ we deduce $Q$ by MP and therefore $C_n$ by (2); thus $Q Z_{n + 1} , Z_{n + 1} \vdash C_n$ and so $Q Z_{n + 1} \vdash Z_{n + 1} \supset C_n = C_{n + 1}$ by DT.]
\medbreak 

 The theorem now follows from (1) and (3).  
\end{proof}

\medbreak 

Equivalently (by DT and MP) $Q Z_n \supset \CC (Z_N, \dots , Z_0)$ is a theorem of IPC. 

\medbreak 

Although we shall not need the following complementary pair of properties, we include them at little cost; they amount to a de Morgan law. The one property is that if $N \in \mathbb{N}$ then 
$$Q Z_0 \vee \cdots \vee Q Z_N \vdash \CC (Z_N, \dots , Z_0)$$
which is a fairly routine inductive consequence of Theorem \ref{Elimination} and Theorem \ref{vdash}. The other property is the following opposite deduction and perhaps calls for a more detailed argument. 

\medbreak 

\begin{theorem} 
If $N \in \mathbb{N}$ then $\CC (Z_N, \dots , Z_0) \vdash Q Z_0 \vee \cdots \vee Q Z_N$. 
\end{theorem} 

\begin{proof} 
For convenience, write $C_N = \CC (Z_N, \dots , Z_0)$ and $D_N = Q Z_0 \vee \cdots \vee Q Z_N$. Plainly, $C_0 = Z_0 \supset Q = Q Z_0 = D_0$. Now take $C_N \vdash D_N$ as inductive hypothesis. The three assumptions $C_{N + 1}, D_N \supset Q Z_{N + 1}, Z_{N + 1}$ yield the following successive deductions: $C_N$ (from $C_{N + 1} = Z_{N + 1} \supset C_N$ and $Z_{N + 1}$ by MP); $D_N$ (from $C_N$ by the inductive hypothesis); $Q Z_{N + 1}$ (from $D_N \supset Q Z_{N + 1}$ and $D_N$); $Q$ (from $Q Z_{N + 1} = Z_{N + 1} \supset Q$ and $Z_{N + 1}$). Thus 
$$C_{N + 1}, D_N \supset Q Z_{N + 1}, Z_{N + 1} \vdash Q$$
and so by two applications of DT we conclude 
$$C_{N + 1} \vdash (D_N \supset Q Z_{N + 1}) \supset Q Z_{N + 1} = D_{N + 1}.$$
\end{proof}

\medbreak 

In order to introduce $Q$-tableaux, we find it convenient to assume knowledge of the theory of tableaux for signed formulas in the classical Propositional Calculus, for details of which we refer to the classic treatise [5]. IPC formulas of Type A have the form $\alpha = F (X \supset Y)$ with $\alpha_0 = T X$ and $\alpha_1 = FY$ as direct consequences, while IPC formulas of type B have the form $\beta = T (X \supset Y)$ with $\beta_0 = F X$ and $\beta_1 = T Y$ as alternative consequences; symbolically, 
$$ \; \; \; \; \; \; \; \; \; \; \; \; \; \; \; \;  \frac{\alpha}{\alpha_0} \; \; \; \frac{\beta}{\beta_0 \; | \; \beta_1}.$$
$$\alpha_1$$
If the IPC formula $Z$ is a tautology (true in all Boolean valuations) then $F Z$ starts a signed tableau that is closed in the sense that each of its branches contains a conjugate pair $TW$, $FW$ of signed formulas. For all of this and much more, see especially Chapter II of [5]. 

\medbreak 

Now, fix a choice of IPC formula $Q$. Let $Z$ be an IPC tautology and construct a closed signed tableau $\mathcal{T}$ starting from $F Z$; in the construction, do not abbreviate $W \supset Q$ as $Q W$. Replace each node in $\mathcal{T}$ of the form $F W$ by $Q W$ and replace each node in $\mathcal{T}$ of the form $T W$ by $QQ W$. The result is a tableau $\mathcal{T}_Q$ starting from $Q Z$ with the following branching rules: \\

$\alpha = Q (X \supset Y)$ has direct consequences $\alpha_0 = QQ X$ and $\alpha_1 = Q Y$; \\

$\beta = QQ (X \supset Y)$ has alternative consequences $\beta_0 = Q X$ and $\beta_1 = QQ Y$. \\

\noindent
Each branch $\theta$ of $\mathcal{T}_Q$ is a sequence $(Z_N, \dots , Z_0)$ with $Z_0 = Q Z$ and each term $Z_n$ of the form $Q W_n$ or $QQ W_n$ for some IPC formula $W_n$. Each branch $\theta$ of $\mathcal{T}_Q$ is closed in the sense that among its terms is a conjugate pair $Q W$, $QQ W$ for some IPC formula $W$. 

\medbreak 

\noindent
{\bf Remark:} We may instead define a $Q$-tableau for $Z$ as a tableau starting from $Q Z$ with the branching rules displayed above; it was simply easier to import the standard machinery of tableaux for signed formulas.

\medbreak 

Motivated by the construction in [6] for the classical Propositional Calculus, we associate to the IPC formula $Q$ an axiom system $\mathbb{U}_Q$ having the following axiom schemes and inference rules, throughout which $\theta = (Z_N, \dots , Z_0)$ stands for sequences of IPC formulas, each $Z_n$ being of the form $Q W_n$ or $QQ W_n$ for some IPC formula $W_n$, such a sequence $\theta$ being closed precisely when it has a conjugate pair $Q W, QQ W$ among its terms. \\

{\it Axioms}: All formulas $\CC (\theta) = Z_N \supset \cdots \supset Z_0 \supset Q$ for which $\theta = (Z_N, \dots \, Z_0)$ is closed. \\

{\it Rule A}: If $\alpha$ is a term of $\theta$ then from $\CC(\alpha_0, \theta)$ or $\CC(\alpha_1, \theta)$ (separately) infer $D(\theta)$. \\

{\it Rule B}: If $\beta$ is a term of $\theta$ then from $\CC(\beta_0, \theta)$ and $\CC(\beta_1, \theta)$ (together) infer $D(\theta)$. \\

As is the case for their counterpart in the classical Propositional Calculus [6], $Q$-tableaux and their axiom systems associated to IPC formulas $Q$ facilitate a proof that the Implicational Propositional Calculus is complete, as we now proceed to show. 

\medbreak 

\begin{theorem} \label{axioms}
Each axiom of $\mathbb{U}_Q$ is a theorem of IPC. 
\end{theorem} 

\begin{proof} 
Let the sequence $\theta = (Z_N, \dots , Z_0)$ contain both $Q W$ and $QQ W$ as terms: Theorem \ref{vdash} tells us that $QQ W \supset \CC (\theta) \in \mathbb{T}$ and $QQQ W \supset \CC (\theta) \in \mathbb{T}$; the Remark after Theorem \ref{middle} then places $\CC (\theta)$ in $\mathbb{T}$. 
\end{proof} 

\medbreak 

Rule A of $\mathbb{U}_Q$ may be regarded as a derived inference rule for IPC. 

\medbreak 

\begin{theorem} \label{RuleA} 
Let $\theta$ have $\alpha$ as a term. If $\CC(\alpha_0, \theta) \in \mathbb{T}$ or $\CC(\alpha_1, \theta) \in \mathbb{T}$ then $\CC (\theta) \in \mathbb{T}$. 
\end{theorem} 

\begin{proof} 
Theorem \ref{vdash} guarantees that $Q \alpha \supset \CC (\theta) \in \mathbb{T}$. Theorem \ref{extra} parts (${\rm A}_0$) and (${\rm A}_1$) guarantee that $\alpha \supset \alpha_0 \in \mathbb{T}$ and $\alpha \supset \alpha_1 \in \mathbb{T}$; it follows from this by HS that  if $\alpha_0 \supset \CC(\theta) = \CC(\alpha_0, \theta) \in \mathbb{T}$ or $\alpha_1 \supset \CC(\theta) = \CC(\alpha_1, \theta) \in \mathbb{T}$ then $\alpha \supset \CC (\theta) \in \mathbb{T}$. Finally, the Remark after Theorem \ref{middle} places $\CC (\theta)$ in $\mathbb{T}$. 
\end{proof} 

\medbreak 

Rule B of $\mathbb{U}_Q$ may also be seen as a derived inference rule for IPC. 

\medbreak 

\begin{theorem} \label{RuleB} 
Let $\theta$ have $\beta$ as a term. If $\CC(\beta_0, \theta) \in \mathbb{T}$ and $\CC(\beta_1, \theta) \in \mathbb{T}$ then $\CC (\theta) \in \mathbb{T}$. 
\end{theorem} 

\begin{proof} 
Theorem \ref{vdash} guarantees that $Q \beta \supset \CC (\theta) \in \mathbb{T}$. Theorem \ref{extra} part (B) guarantees that $\beta \supset (\beta_0 \vee \beta_1) \in \mathbb{T}$; consequently, if  $\beta_0 \supset \CC(\theta) =  \CC(\beta_0, \theta) \in \mathbb{T}$ and $\beta_1 \supset \CC(\theta) = \CC(\beta_1, \theta) \in \mathbb{T}$ then $(\beta_0 \vee \beta_1) \supset \CC(\theta) \in \mathbb{T}$ by Theorem \ref{Elimination} so that HS yields $\beta \supset \CC(\theta) \in \mathbb{T}$. Finally, the Remark after Theorem \ref{middle} places $\CC(\theta)$ in $\mathbb{T}$. 
\end{proof} 

\medbreak 

Taken together, Theorems \ref{axioms}, \ref{RuleA} and \ref{RuleB} establish that all theorems of $\mathbb{U}_Q$ are provable within IPC. 

\medbreak 

\begin{theorem} \label{UIPC}
Each theorem of $\mathbb{U}_Q$ is a theorem of IPC. 
\end{theorem} 

\begin{proof} 
The set $\mathbb{T}$ of all IPC theorems contains the axioms of $\mathbb{U}_Q$ by Theorem \ref{axioms}; it is closed under Rules A and B according to Theorems \ref{RuleA} and \ref{RuleB}. 
\end{proof} 

Now, let us return to the IPC tautology $Z$ to which we associated a closed $Q$-tableau $\mathcal{T}_Q$ starting from $Z_0 = Q Z$. As each branch $\theta$ of $\mathcal{T}_Q$ is closed, each corresponding $\CC (\theta)$ is an axiom of $\mathbb{U}_Q$. We prune the tableau $\mathcal{T}_Q$ by reversing the steps by which it was formed: pruning $\theta$ applies an inference rule of $\mathbb{U}_Q$ to $\CC(\theta)$ and so results in a theorem of $\mathbb{U}_Q$; the final pruning lays bare the root $Z_0 \supset Q = Q Z \supset Q = QQ Z$ which is then itself a theorem of $\mathbb{U}_Q$. Conclusion: if $Z$ is an IPC tautology, then $QQ Z = (Z \supset Q) \supset Q$ is a theorem of $\mathbb{U}_Q$. 

\medbreak 

It is now a short step to completeness of IPC. 

\begin{theorem} 
IPC is complete. 
\end{theorem} 

\begin{proof} 
Let $Z$ be an IPC tautology and take $Q:= Z$. As we have just seen, $(Z \supset Z) \supset Z$ is a theorem of $\mathbb{U}_Z$ and therefore a theorem of IPC by Theorem \ref{UIPC}. The proof is concluded by an application of MP to $(Z \supset Z) \supset Z$ and the specific Peirce axiom $[(Z \supset Z) \supset Z] \supset Z$. 
\end{proof} 

\bigbreak 

In closing, we note that there are significant differences between the approach to IPC completeness via $Q$-tableaux offered here and the approach via dual $Q$-tableaux offered in [4]. One difference relates to the r\^ole played by the Peirce axiom scheme: Theorems \ref{axioms}, \ref{RuleA} and \ref{RuleB} of the present paper are all concluded by an application of the Peirce axiom scheme in its guise as a weak law of the excluded middle; by contrast, the corresponding results in [4] hinge on various parts of Exercise 6.3 in [2]. Another difference relates to conjunction (`negated') and disjunction: in $Z_N \supset \cdots \supset Z_0 \supset Q$ new terms are added on the left while in $Z_0 \vee \cdots \vee Z_N$ they are added on the right; this difference is reflected in the definition and properties of the corresponding axiom systems. We leave the details of a full comparison to the reader.

\bigbreak

\newpage

\begin{center} 
{\small R}{\footnotesize EFERENCES}
\end{center} 
\medbreak 

[1]  A. Church, {\it Introduction to Mathematical Logic}, Princeton University Press (1956). 

[2] J. W. Robbin, {\it Mathematical Logic - A First Course}, W.A. Benjamin (1969); Dover Publications (2006).

[3] P.L. Robinson, {\it The Peirce axiom scheme and suprema}, arXiv 1511.07074 (2015). 

[4] P.L. Robinson, {\it Implicational Propositional Calculus - Tableaux and Completeness}, arXiv 1512.00091v2 (2015).  

[5] R.M. Smullyan, {\it First-Order Logic}, Springer-Verlag (1968); Dover Publications (1995). 

[6] R.M. Smullyan, {\it A Beginner's Guide to Mathematical Logic}, Dover Publications (2014).

\medbreak

\end{document}